\documentclass[a4paper,11pt]{amsart}

\usepackage{amsmath,amsfonts,amssymb,amsthm}
\usepackage{latexsym,bm,graphicx}

\usepackage{color}
\title[Ricci Curvature]{On a New Definition of  Ricci Curvature on Alexandrov Spaces}
\author{Hui-Chun Zhang}
\address{Department of Mathematics\\  Sun Yat-sen University\\ Guangzhou 510275\\ E-mail address: zhhuich@mail2.sysu.edu.cn}
\author{ Xi-Ping Zhu}
\address{Department of Mathematics\\  Sun Yat-sen University\\ Guangzhou 510275\\ E-mail address: stszxp@mail.sysu.edu.cn}
\newtheorem{thm}{Theorem}[section]
\newtheorem{prop}[thm]{Proposition}
\newtheorem{lem}[thm]{Lemma}

\newtheorem{cor}[thm]{Corollary}

\newtheorem{conj}[thm]{Conjecture}
\newtheorem{prob}[thm]{Open Problem}

\theoremstyle{definition}
\newtheorem{defn}[thm]{Definition}

\theoremstyle{remark}

\numberwithin{equation}{section}

\newcommand{\ls}{\leqslant}
\newcommand{\gs}{\geqslant}
\newcommand{\wa}{\widetilde\angle}
\newcommand{\wt}{\widetilde}

\newcommand{\ka}{\kappa}

\newcommand{\R}{\mathbb{R}}

\newcommand{\ip}[2]{\left<{#1},{#2}\right>}

\newcommand{\be}[2]{\begin{#1}{#2}\end{#1}}

\begin{document}

\maketitle
\begin{abstract}  Recently, in \cite {ZZ}, a new definition for lower Ricci curvature bounds on Alexandrov spaces was introduced by the authors.  In this article, we  extend our research to summarize the  geometric and analytic results under this Ricci condition. In particular, two new results, the rigidity result of Bishop-Gromov volume comparison and Lipschitz continuity of heat kernel,  are obtained.
\end{abstract}
\tableofcontents
\setcounter{tocdepth}{1}
\section{Introduction}
A complete metric space $(X,|\cdot,\cdot|)$ is called to be a geodesic space if, for any two points $p,q\in X$, the distance $|pq|$ is realized as the length of a rectifiable curve connecting $p$ and $q$. Such distance-realizing curves, parameterized by arc-length, are called (minimal) geodesics.

A geodesic space $(X,|\cdot,\cdot|)$ is said to have curvature $\gs k$ in an open set
$U\subset X$ if
for  each quadruple $(p;a, b,c)\subset U $,  \be{equation}{\wa_k apb+\wa_k bpc+\wa_k cpa\ls 2\pi,} where $\wa_k apb,
\wa_k bpc$ and $\wa_k cpa$ are the comparison angles in the
$k-$plane. That is,  $\wa_k apb$ is the angle at $\bar p$ of a triangle $\triangle\bar a\bar p\bar b$ with side lengths $|\bar a\bar p|=|ap|$, $|\bar p \bar b|=|pb|$  and  $|\bar a\bar b|=|ab|$ in the
$k-$plane. See \cite{BGP} for others equivalent definitions for curvature $\gs \ka$.

 A geodesic space $X$ is  called to be an\emph{ Alexandrov space
with curvature bounded from below locally} (for short, we say $X$ to
be an \emph{Alexandrov space}), if it is locally compact and any
point $p\in M$ has an open neighborhood $U_p\ni p$ and a number $k_{p}\in \mathbb{R}$ such that $X$ has
curvature  $\gs k_{p}$ in $U_p$.  We say that $X$ has curvature $\gs k$ if the previous statement holds  with $k_{p}=k$ for all $p$. It was proved in \cite{BGP}  that $X$ having curvature $\gs k$ implies that (1.1) holds for all quadruples $(p;a, b,c)$ in $X$.

Basic examples of Alexandrov spaces are listed as follows:\\
\indent (1)\indent Riemanian manifolds without or with boundary. A Riemanian manifold has curvature $\gs k$ in an open convex set $U$ if and only if  its sectional curvature $\gs k$ in $U.$\\
\indent (2)\indent Convex polyhedra. The boundary of a convex body (compact convex
set with nonempty interior) in Euclidean spaces has curvature $\gs0$.\\
\indent (3)\indent Let $M$ and $N$ be two Alexandrov spaces. Then the direct product space $M\times N$ is an Alexandrov space.\\
\indent (4)\indent Let $M$ be an Alexandrov space and let the group $G$ act isometrically on $M$, (not necessarily acting free). Then the quotient  space $\overline{M/G}$ is an Alexandrov space.\\
\indent (5)\indent
Let $X$ be a complete metric space of diameter $\leq \pi$. The suspension and cone over $X$ are defined as follows.\\
 \indent\indent (i)\indent The \emph{suspension} over $X$  is the quotient space
 $\ S(X)=X\times [0,\pi]/\sim$, where $(x_1,a_1)\sim (x_2,a_2)\Leftrightarrow a_1=a_2=0 \ {\rm or}\ a_1=a_2=\pi$ with the
 canonical metric
$$\cos|\bar x_1\bar
 x_2|=\cos a_1\cos a_2+\sin a_1\sin a_2\cos|x_1x_2|_X,$$where $\bar
 x_1=(x_1,a_1),\ \bar
 x_2=(x_2,a_2).$\\
\indent\indent (ii)\indent The \emph{cone} over $X$ is the quotient space
 $\ C(X)=X\times [0,\infty)/\sim$, where $(x_1,a_1)\sim (x_2,a_2)\Leftrightarrow a_1=a_2=0.$  The metric of the cone is defined from the
 cosine formula,$$|\bar x_1\bar
 x_2|^2=a_1^2+a_2^2-2a_1a_2\cos|x_1x_2|_X,$$where $\bar
 x_1=(x_1,a_1),\ \bar
 x_2=(x_2,a_2).$\\
 \indent If $X$ is an Alexandrov space with curvature $\gs 1$, then the \emph{suspension} over $X$ has curvature $\gs1$, and
the \emph{cone} over $X$ has curvature $\gs 0$.

 The seminal paper \cite{BGP} and the 10th chapter
in the text book \cite{BBI} provide excellent introductions to Alexandrov geometry.\\

One of the major concepts in Riemannian geometry is ``curvature'', including ``sectional curvature'' and ``Ricci curvature''.
Alexandrov spaces admit the notion of ``lower bounds of sectional curvature''.
 However,  many fundamental results in Riemannian geometry, such as Bonnet-Myers' theorem,  Bishop-Gromov relative volume comparison theorem, Cheeger-Gromoll splitting theorem, Cheng's maximal diameter theorem and Li-Yau's gradient estimates, are established on  \emph{Ricci} curvature.
Thus, a natural question is to give a notion of ``lower bounds of Ricci curvature'' for Alexandrov spaces.  Such a generalization should  satisfy the following properties:\\
\indent (1)\indent it reduces to the usual one for smooth Riemannian manifolds;\\
\indent (2)\indent it admits interesting geometric results on Alexandrov spaces with ``Ricci curvature bounded below''.

In the last few years,  several notions for the ``Ricci curvature bounded below'' on general metric spaces appeared.  Sturm \cite{S-acta} and Lott-Villani  \cite{LV-ann, LV-jfa}, independently, introduced  a definition of ``Ricci curvature bounded lower" for a metric measure space $(X,d,m)$\footnote{A metric measure space  $(X,d,m)$ is a metric space $(X,d)$ equipped a Borel measure $m$.}, by utilizing convexity of some functionals  on the associated  $L^2-$Wasserstein spaces (the space of all probability measures on $X$ with finite second moment). They call it the curvature-dimension condition, denoted by $CD(n,k)$ with  $n\in(1,\infty]$ and $k\in\mathbb{R}.$ Meanwhile, Sturm in \cite{S-acta} and Ohta in \cite{O-mcp} introduced another definition
of ``Ricci curvature bounded below'' for metric measure spaces, the measure
contraction property $MCP(n,k)$, which is a slight modification of a property introduced earlier by Sturm in \cite{S-diff} and in a similar form by Kuwae and
Shioya in \cite{KS-mcp, KS-sob}. The condition $MCP(n,k)$ is indeed an infinitesimal
version of the Bishop-Gromov relative volume comparison condition. In an $n-$dimensional Riemannian manifold $M$, both $CD(n,k)$ and $MCP(n,k)$ are equivalent to the usual condition of Ricci curvature  $\gs k$.

It is obvious that both $CD(n,k)$ and $MCP(n,k)$ make sense for Alexandrov spaces associated to their Hausdorff measures. Moreover, we know from \cite{S-acta} that $CD(n,k)$ implies $MCP(n,k)$ in Alexandrov spaces, due to  non-branching property of  Alexandrov spaces. On $n-$dimensional Alexandrov spaces, Kuwae and Shioya in \cite{KS-lap} introduced another infinitesimal
version of the Bishop-Gromov  volume comparison condition (see also \cite{KS-inf}), denoted by $BG(k)$. Indeed, $MCP(n,k)$ is equivalent to $BG(k)$ on an $n-$dimensional Alexandrov space (see  for example \cite{O-mcp}).

As mentioned before, one expects that  a good definition of ``Ricci curvature bounded below'' on Alexandrov spaces should allow  as many geometric
consequences as manifold case. Note that definition of $MCP(n,k)$ (or $BG(k)$) is an infinitesimal version of Bishop-Gromov relative volume comparison and the condition $CD(n,k)$ implies $MCP(n,k)$ in Alexandrov spaces. Thus, there holds Bishop-Gromov relative volume comparison theorem under $CD(n,k)$ or $MCP(n,k)$ (or $BG(k)$). It was shown in \cite{S-acta}, \cite{LV-ann} and \cite{O-mcp} that Bonnet-Myers' theorem also holds under $CD(n,k)$ or $MCP(n,k)$ (or $BG(k)$). But, up to now, our knowledge on the geometric consequences under these Ricci conditions is still very limit.

 Note that every finitely dimensional norm space ($V^n, \|\cdot\|_p$) satisfies $CD(n,0)$
for $p>1$ (see, for example, page 892 in [V]). More generally,  curvature-dimension condition on a Finsler
manifolds is equivalent to the Finsler Ricci curvature condition (see Ohta \cite{O-finsler}). Clearly, one does not expect a type of Cheeger-Gromoll splitting theorem on Finsler manifolds.  Therefore, Cheeger-Gromoll splitting theorem  is generally not true under $CD(n,0)$ for general metric measure spaces. In \cite{KS-lap,KS-wei}, Kuwae and Shioya established a topological version of Cheeger-Gromoll splitting theorem on Alexandrov spaces under $BG(0)$.

 In \cite{ZZ}, the authors introduced a new definition for lower bounds of Ricci curvature on Alexandrov spaces. We have shown in \cite{ZZ} that the new definition implies the curvature-dimension condition and there hold Cheeger-Gromoll splitting theorem and maximal diameter theorem on Alexandrov spaces under the new notion of Ricci curvature. In this paper, we extend our research to summarize the geometric and analytic results under this Ricci condition. In particular,  two new results, the rigidity result of Bishop-Gromov volume comparison (see Theorem 3.7) and Lipschitz continuity of heat kernel (see Theorem 5.14),  are obtained.

{\bf Acknowledgements}\ \    We are grateful  to Dr. Qintao Deng for
helpful discussions. We are also like to thank Professor T. Shioya
for his helpful comments. The second author is partially supported
by NSFC 10831008 and NKBRPC 2006CB805905.

\section{Definitions of Ricci curvature}

\subsection{Ricci curvature on smooth  manifolds} To illustrate the idea of our definition of lower Ricci curvature bounds on Alexandrov spaces,  we recall  some equivalent conditions for   Ricci curvature  on smooth Riemannian manifolds.

Let $M^n$ be an $n-$dimensional Riemannian manifold and  let $R$ be
the Riemannian curvature tensor. Fix a shortest geodesic
$\gamma(t)$, $t\in(-\epsilon,\epsilon),$ and an othonormal basis
$\{e_1, e_2,\cdots,e_n =\gamma'(0) \}$ at $ p =\gamma(0)$. We extend
them to  an orthonormal frame $\{e_1(t), e_2(t),\cdots,e_n(t) \}$ on
$\gamma(t)$ by parallel
translation. The sectional curvature on  $2-$plane $P_{ij}\subset T_pM^n$, spanned by vectors $e_i$ and $e_j$, is defined by $$sec(P_{ij})=R(e_i,e_j,e_j,e_i).$$   Fix $t_0\in(-\epsilon,\epsilon),$ and let $P_{in}(t_0)$ be the $2-$plane in $T_{\gamma(t_0)}M^n$ spanned by vectors $\gamma'(t_0)$ and $e_i(t_0)$. Then, by the second variation formula of arc-length, the condition $sec(P_{in}(t_0))\gs \ka_i(t_0)$ (for some real number $\ka_i(t_0)$) is equivalent to the following geometric property:\\
\indent for $x=\gamma(t_0)$ and any $\delta>0$, there exists $\eta_0>0$ with $(t_0-\eta_0,t_0+\eta_0) \subset(-\epsilon,\epsilon)$ such that for any $y= \gamma(t')$ with $t'\in (t_0-\eta_0,t_0+\eta_0)$,
\be{equation}{\be{split}{d\big(\exp_x(\varepsilon ae_i(t_0)),&\exp_y(\varepsilon b e_i(t'))\big)\ls d(x,y)\\ &+\Big(\frac{(b-a)^2}{2\cdot d(x,y)}+\frac{(\ka_i(t_0)+\delta) \cdot d(x,y)}{6}(a^2+ab+b^2)\Big)\cdot \varepsilon^2+o(\varepsilon^2)}}
as $\varepsilon\to0^+,$ for all $a,b\gs0.$

The Ricci curvature $\gs k$ at $x=\gamma(t_0)$ is equivalent to having $\ka_1(t_0),\  \ka_2(t_0),\cdots, \ka_{n-1}(t_0)$ in (2.1)  with  \begin{equation}\ka_1(t_0)+\ka_2(t_0)+\cdots+\ka_{n-1}(t_0)\gs k,\end{equation}for all geodesics $\gamma$ passing through $x$.

One can also characterize the condition of Ricci curvature bounded below via the Jacobian fields along geodesics. To see this, let $\phi(x)$ be a $C^3$ function defined in a neighborhood of a given shortest geodesic $\gamma(t): (-\epsilon,\epsilon)\to M^n$, and consider the map $F_t(x):=\exp_x(t\nabla \phi(x))$. Then $Jac(F_t)(x)$ can be described as the determinant of a  matrix $\emph{\textbf{J}}(t)$ which solves the Jacobi equations $$\emph{\textbf{J}}''(t)+\emph{\textbf{RJ}}(t)=0$$
with initial data $\emph{\textbf{J}}(0)=\emph{\textbf{Id}}$ and $\emph{\textbf{J}}'(0)=Hess_p\phi$, where $\emph{\textbf{R}}(t)=R(\gamma'(t),e_i,e_j,\gamma'(t))$.
Setting $U(t)=\emph{\textbf{J}}'\cdot \emph{\textbf{J}}^{-1}$ and $\mathcal{J}={\rm det}\emph{\textbf{J}}$, we have \be{equation}{\frac{d}{dt}{\rm tr} U+{\rm tr} U^2+Ric(\gamma',\gamma')=0
} and, by Cauchy-Schwarz inequality,
\be{equation}
{\frac{d^2}{dt^2}\log \mathcal{J}+\frac{1}{n}\Big(\frac{d}{dt}\log \mathcal{J}\Big)^2+Ric(\gamma',\gamma')\ls0.
}

Denote by  $\mathcal P_2(M^n,d_W,{\rm vol})$  the subset of
$L^2-$Wasserstein space which consists of absolutely continuous
probability measures with respect to {\rm vol}. $M^n$ is said to
possess displacement $k-$convexity if the functional (or called
entropy) \be{equation*} {Ent(\mu)=\int_{M^n}
\frac{d\mu}{dx}\cdot\log\frac{d\mu}{dx} d{\rm vol}(x)
 } is $k-$convex in  $\mathcal P_2(M^n,d_W,{\rm vol})$.

By integrating equation (2.4), Cordero-Erausquin, McCann and
Schmuckenschl\"{a}ger in \cite{CMS} proved that Riemannian manifolds
with nonnegative Ricci curvature possess displacement $0-$convexity.
In particular, the displacement $0-$convexity implies a generalized
Brunn-Minkowski inequality which states that the function $$t\to
{\rm vol}^{1/n}(A_t)\qquad t\in[0,1]$$is concave, where measurable
sets $A_t$ are defined by
$$A_t=\{x_t\in M: \exists\ x_0\in A_0, x_1\in A_1\ {\rm such\ that}\
|x_0x_t|=t|x_0x_1|,\ |x_tx_1|=(1-t)|x_0x_1|\}.$$  Later in
\cite{RS}, this displacement convexity  was extended by von Renesse
and Sturm to displacement $k-$convexity for Riemannian manifolds
with  Ricci curvature bounded below by $k$. In fact, the
displacement $k-$convexity gives an equivalent definition for Ricci
curvature bounded below by $k$ on Riemannian manifolds.

If we denote $A_p(r,\xi)$ the density of the Riemannian measure on
$\partial B_p(r)$ induced from the Riemannian metric on $M^n$, then
by the classical Bishop comparison theorem (see, for example
\cite{Chavel}), the condition $Ric(M^n)\gs k$ implies that the
function
$$\frac{A_p(r,\xi)}{(s_k(r/\sqrt{n-1}))^{n-1}}$$ is non-increasing
in $(0,c(\xi))$ for all $\xi\in \Sigma_p$, where  $s_k(t)$ is the
solution of $\chi''(t)+k\cdot\chi(t)=0$ with  $\chi(0)=1,\chi'(0)=1$
and $$c(\xi):=\sup\{t>0|\ |p\exp_p(t\xi)|=t\}.$$ On the other hand,
given a direction $\xi\in \Sigma_p$, there holds
$$\frac{A_p(r,\xi)}{A_p(2r,\xi)}=\frac{1}{2^{n-1}}\cdot\big(1+Ric(\xi,\xi)\cdot
r^2\big)+O(r^3).$$  Then it is no hard to show that the inequality
$$\frac{A_p(r,\xi)}{A_p(2r,\xi)}\gs \Big(\frac{s_k(r/\sqrt{n-1})}{s_k(2r/\sqrt{n-1})}\Big)^{n-1}$$ implies the condition $Ric(\xi,\xi)\gs k$. Therefore, for an $n-$dimensional Riemannian manifold $M^n$, its Ricci curvature bounded below by $k$  if and only if for all $p\in M^n$, the function $A_p(r,\xi)/(s_k(r/\sqrt{n-1}))^{n-1}$ is non-increasing in $(0,c(\xi))$ for all $\xi\in \Sigma_p$.

We can now summarize the equivalent conditions of Ricci curvature bounded below in the following proposition.
\begin{prop}On an $n-$dimensional Riemannian manifold $M^n$, the following five conditions  are  equivalent:\\
\indent{\rm (i)}\ $Ric(M^n)\gs k$;\\
\indent{\rm (ii)}\ Bochner formula: for each $C^3$ function $f$,
\be{equation*}{\frac{1}{2}\Delta|\nabla f|^2\gs|Hess f|^2+\ip{\nabla f}{\nabla\Delta f}+k|\nabla f|^2\gs \frac{(\Delta f)^2}{n}+\ip{\nabla f}{\nabla\Delta f}+k|\nabla f|^2;
}
\indent{\rm (iii)}\ displacement $k-$conexity $($see \cite{RS}$)$;\\
\indent{\rm (iv)}\ Bishop comparison property $($see, for example \cite{Chavel}$)$:  for all $p\in M^n$, the function $$\frac{A_p(r,\xi)}{(s_k(r/\sqrt{n-1}))^{n-1}}$$ is non-increasing in $(0,c(\xi))$ for all   $\xi\in \Sigma_p$;\\
\indent{\rm (v)}\ parallel transportation explanation by $(2.1)$ and $(2.2)$.
\end{prop}
\subsection{Ricci curvature bounded below on singular spaces} Each equivalent condition is Proposition 2.1 can be used to define a notion of Ricci curvature
bounded below on suitable singular spaces. In \cite{LY}, Lin and Yau
used the condition (ii) to define the lower bounds for Ricci
curvature on locally finite graphs. In this subsection we will
recall Lott-Sturm-Villani's curvature-dimension condition and
Ohta-Sturm's $MCP$ condition (Kuwae-Shioya's $BG$ condition), which
are associated to the above conditions (iii) and (iv), respectively.

\subsubsection{Curvature-dimension condition $CD(n,k)$}
 Let $(X,d,m)$ be a metric measure space. Let us limit ourselves to the case  that $(X,d)$ is a non-branching complete separable geodesic space\footnote{Lott-Villani and Sturm defined curvature dimension condition on  general metric measure spaces.}.
Denote the $L^2-$Wasserstein space by $\mathcal P_2(X,d_W)$ and its subspace consisting of $m-$absolutely continuous probability measures  by $\mathcal P_2(X,d_W,m)$.

 Set $$U_n(r)=nr(1-r^{-1/n})$$for $N< \infty$ and $U_n(r)=r\ln r$ for $n=\infty.$ Recall that in the above (iii) in Proposition 2.1, the function $U_\infty(r)$ is used to define the  functional $Ent.$

 Given  $k\in\mathbb{R}$, $n\in(1,\infty]$, $t\in[0,1]$ and two points $x,y\in X$, the function $\beta_t^{(k,n)} $ is defined as follows:\\
   \indent (1)\indent If $0<t\ls 1$, then \begin{equation*}\beta_{t}^{(k,n)}(x,y):=
 \begin{cases}\exp\big(\frac{k}{6}(1-t^2)\cdot d^2(x_0,x_1)\big)& {\rm if }\ n=\infty,\\ \infty & {\rm if }\ n<\infty,\ k>0\ \ {\rm and}\  \alpha\gs \pi,\\  \big(\frac{\sin(t\alpha)}{t\sin\alpha}\big)^{n-1}& {\rm if }\ n<\infty,\ k>0\ \ {\rm and}\ \alpha\in[0, \pi),\\
1& {\rm if }\ n<\infty, \ k=0, \\
 \Big(\frac{\sinh(t\alpha)}{t\sinh\alpha}\Big)^{n-1}& {\rm if }\ n<\infty,\ k<0,\end{cases}
 \end{equation*}  where $\alpha=d(x,y)\cdot\sqrt{|k|/(n-1)}$.\\
\indent (2)\indent $\beta_0^{(k,n)}(x,y)=1.$

The curvature-dimension condition $CD(n,k)$ is a kind of convexity for the functionals defined by $U_n$.

\begin{defn}(see 29.8 and 30.32 in \cite{V})\
Given  $k\in\mathbb{R}$ and $n\in(1,\infty]$,  the metric measure space $(X,d,m)$ is said to satisfy the \emph{curvature-dimension condition }$CD(n,k)$ if for
each pair  $\mu_0,\mu_1\in\mathcal P_2(X,d_W,m)$ with compact support and ${\rm Supp}(\mu_i)\subset{\rm Supp}(m)$, $i=0,1$, writing their Lebesgue decompositions
with respect to $m$ as $\mu_0=\varrho_0\cdot m$ and ¦Ì$\mu_1=\varrho_1\cdot m$, respectively, then there exist an optimal coupling $q$ of $\mu_0$ and $\mu_1$,
and a geodesic path\footnote{constant-speed shortest curve defined on $[0,1]$.} $\mu_t: [0, 1]\rightarrow \mathcal P_2(X,d_W)$ connecting $\mu_0$ and $\mu_1$, so that for  all $t\in [0, 1]$, we have
\begin{equation}\begin{split}
H_m(\mu_t)\ls& (1-t)\int_{X\times
X}\frac{\beta^{(k,n)}_{1-t}(x,y)}{\rho_0(x)}\cdot U_n\Big(\frac{\varrho_0(x)}{\beta^{(k,n)}_{1-t}(x,y)}\Big)dq(x,y)\\
&\ +t\int_{X\times
X}\frac{\beta^{(k,n)}_{t}(x,y)}{\rho_1(y)}\cdot U_n\Big(\frac{\varrho_1(y)}{\beta^{(k,n)}_{t}(x,y)}\Big)dq(x,y)
\end{split}\end{equation} where $H_m(\mu): \mathcal P_2(X,d_W)\to\mathbb{R}$ is the functional
$$H_m(\mu):=\int_XU_n(\varrho)dm+\lim_{r\to\infty}\frac{U_n(r)}{r}\cdot\mu_s(X)$$ and $\mu$ has Lebesgue decomposition with respect to $m$ as $\mu=\rho\cdot m+\mu_s$.
\end{defn}

Let $M^n$ be a Riemannian manifold with Riemannian distance $d$ and Riemannian volume ${\rm vol}$.    The equivalence between the    metric measure space $(M^n,d,{\rm vol})$ satisfying $CD(n,k)$ and  the Riemannian manifold $M^n$ having Ricci curvature $\gs k$ is proved by Lott-Villani in \cite{LV-ann,LV-jfa} and von Renesse-Sturm in \cite{RS,
S-convex}.  The idea of  the proof can be described as follows. Take $\mu_0,\mu_1\in\mathcal P_2(M^n,d_W,{\rm vol})$,  there exists a function $\varphi: M^n\to\R$ such that $$\Gamma(t)=(F_t)_*\mu_0$$ forms a geodesic path in $P_2(M^n,d_W)$ connecting $\Gamma(0)=\mu_0$ and $\Gamma(1)=\mu_1$, where $
F_t(x)=\exp_x(-t\nabla \varphi(x))$ for $\mu_0$-a.e. $x\in M^n$ and $t\in[0, 1]$ (see \cite{CMS}). By integrating (2.4),  one can prove that the condition  $Ric\gs k$ implies $CD(n,k).$ Conversely, $CD(n,k)$ implies a Brunn-Minkowksi inequlity, and hence Bishop-Gromov relative volume comparison. Therefore, by Proposition 2.1(iv),  the condition $CD(n,k)$ implies Ricci curvature $\gs k$ on Riemannian manifolds.

\subsubsection{Measure contraction property $MCP(n,k)$ and $BG(k)$}
Let $X$ be a geodesic space. Denote $\Gamma$ by the set of geodesic paths  in $X$ and define the evaluation map $e_t:\Gamma\to X$ by $e_t(\gamma)=\gamma(t)$. A dynamical transference plan $\Pi$ is a Borel probability measure on $\Gamma.$
\be{defn}{(see \cite{O-mcp})  For $n,k\in\mathbb{R}$, $(X,d,m)$ is said to satisfy the condition $MCP(n,k)$ if for any point $x\in X$ and measurable set $A$ with finite positive measure, there exists a geodesic path $\mu_t$ in $\mathcal P_2(X,d_W)$, associated to a dynamical transference plan $\Pi$ , such that $\mu_0=\delta_x$, $\mu_1=[m(A)]^{-1}\cdot m|_A$ and for every $t\in[0,1]$,\be{equation*}{
d\mu_t\gs(e_t)_*\Big(t\Big\{\frac{s_k(t\ell(\gamma))/\sqrt{n-1}}{s_k(\ell(\gamma))/\sqrt{n-1}}\Big\}^{n-1}\cdot m(A)\cdot d\Pi(\gamma)\Big)
}where $\ell(\gamma)$ is the length of $\gamma$.
}
Roughly speaking, $MCP(n,k) $ is the special case of $CD(n,k)$ where $\mu_0$ is degenerated to a Dirac mass in Definition 2.2. We remark that there is another form of this definition  via Markov kernel (see \cite{S-acta}).

For an $n-$dimensional Riemannian manifold $M^n$, Ohta in \cite{O-mcp} proved that $M^n$ satisfies $MCP(n,k)$ if and only if its Ricci curvature $\gs k$. Let us describe his proof as follows. Let $p\in M^n$ and denote $C_p$ to be the cut locus of $p$. Consider the map $\Phi_{p,t}:\ M^n\backslash C_p\to M^n,$ $0<t\ls 1,$ $$\Phi_{p,t}(x):= {\rm the\ point}\ y\ {\rm in \ geodesic\ from\ }p\ {\rm to}\ x\ {\rm such\ that}\ |py|=t|px|.$$
$M^n$ satisfies  $MCP(n,k)$ if and only if  for any $p\in M^n$, the following property holds (see \cite{O-mcp}):\\
\indent for any $  x\in M^n$ and $0<t\ls 1$, \be{equation}{
d(\Phi_{p,t,*}{\rm vol})(x)\gs
t\Big(\frac{s_k(t\ell(\gamma)/\sqrt{n-1})}{s_k(\ell(\gamma)/\sqrt{n-1})}\Big)^{n-1}\cdot
d{\rm vol}(x). } Therefore, $MCP(n,k)$ is indeed an infinitesimal
version of the Bishop-Gromov relative volume comparison condition.
Clearly, the inequality (2.6) implies
$$\frac{A_p(r,\xi)}{A_p(2r,\xi)}\gs
\Big(\frac{s_k(r/\sqrt{n-1})}{s_k(2r/\sqrt{n-1})}\Big)^{n-1}$$ for
all $\xi\in \Sigma_p$, hence $Ric(\xi,\xi)\gs k.$ Conversely, given
any measure set $A\subset M^n$, one has $$\Phi_{p,t,*}{\rm
vol}(A)=\int_{\exp^{-1}_p(A\backslash C_p)}t\cdot
A_p(tr,\xi)drd\xi.$$  Then one gets from that Proposition 2.1 (i)
and (iv) that the condition $Ric(\xi,\xi)\gs k$ implies the
inequality (2.6).

In \cite{KS-lap}, by using inequality  (2.6),  Kuwae and Shioya intruducted an infinitesimal Bishop-Gromov condition, called by $BG(k)$, on Alexandrov spaces. For an $n-$dimensional Alexandrov space with its Hausdorff measure ${\rm vol}$, $BG(k)$ is equivalent to the condition $MCP(n,k)$ (see \cite{O-mcp,KS-lap}).
\subsection{Ricci curvature on Alexandrov spaces}
Let $ M$ be an $n-$dimensional Alexandrov space and  $p\in M$.  $T_p$ and $\Sigma_p$ are the tangent cone and the space of directions. We denote by $C_p$ the \emph{cut locus} to $p$, i.e., the set of points $x\in M$ such that any geodesic from $p$ to $x$, denoted by $\gamma_{px}$, does not extend beyond $x$. It was shown that $C_p$ has $n-$dimensional Hausdorff  measure zero for any $p\in M$ (see \cite{OS-rie}). Denote  by $W_p=M\backslash C_p$. For any $q\in W_p$, the geodesic $\gamma_{pq}$ is unique.

The exponential map $\exp_p: T_p\rightarrow M$ is defined as
follows.  For any  $v\in T_p$,  $\exp_p(v)$ is a point on some
quasi-geodesic (see \cite{Pet-con,PP} for definition of
quasi-geodesic)  starting point $p$ along $v/|v|\in \Sigma_p$ with
$|p\exp_p(v)|=|v|$. Denote by $\log_p:=\exp^{-1}_p.$

 Let $\gamma:\ [0,\ell)\rightarrow M$ be a  geodesic. Without loss of generality, we may assume that a neighborhood  $U_\gamma$ of $\gamma$ has curvature $\gs k_0$ for some $k_0\ls0$.

 From Section 7 in \cite{BGP}, the tangent cone $T_{\gamma(t)}$ at an interior point $\gamma(t) \ (t \in (0,\ell))$ can be split  into a direct  product. We denote \begin{equation*}\begin{split}L_{\gamma(t)}&=\{\xi\in T_{\gamma(t)}\ |\ \ip{\xi}{\gamma^+(t)}=\ip{\xi}{\gamma^-(t)}=0\},\\
\Lambda_{\gamma(t)}&=\{\xi\in \Sigma_{\gamma(t)}\ |\
\ip{\xi}{\gamma^+(t)}=\ip{\xi}{\gamma^-(t)}=0\},
\end{split}\end{equation*}
where $$\gamma^{\pm}(t):=\lim_{h\to0^+}\frac{1}{h}\cdot\log_{\gamma(t)}\gamma(t\pm h).$$
In \cite{Pet-para}, Petrunin proved the following second  variation formula of arc-length.
\begin{prop}$($Petrunin \cite{Pet-para}$)$\indent
Given any two point $q_1,q_2\in\gamma$, which are not end points,
and any  sequence $\{\varepsilon_j\}^\infty_{j=1}$
with $\varepsilon_j\to 0$ and $\varepsilon_j\gs\varepsilon_{j+1}$, there exists a subsequence
$\{\wt\varepsilon_j\}\subset \{\varepsilon_j\}$ and an isometry $T:\
L_{q_1}\rightarrow\ L_{q_2}$ such that
\begin{equation}\begin{split}|\exp_{q_1}(\wt\varepsilon_j u),\ \exp_{q_2}(\wt\varepsilon_j Tv)|\ls & |q_1q_2|+ \frac{|uv|^2}{2|q_1q_2|}\cdot\wt\varepsilon^2_j\\&-\frac{k_0\cdot|q_1q_2|}{6}\cdot\big(|u|^2+|v|^2+\langle u,v\rangle\big)\cdot\wt\varepsilon_j^2+o(\wt\varepsilon_j^2)\end{split}\end{equation}for any $u,v\in L_{q_1}.$
\end{prop}
This proposition is similar to the equation (2.1) in smooth case.
Based on the second variation formula of arc-length, we can propose a condition which resembles the lower bounds for the radial curvature along the geodesic $\gamma$.

  Let $M$ be an $n-$dimensional Alexandrov space without boundary.

\begin{defn} Let $\sigma(t):(-\ell,\ell)\to M$ be a geodesic  and $\{g_{\sigma(t)}(\xi)\}_{-\ell<t<\ell}$ be a  family of  functions on $\Lambda_{\sigma(t)}$ such that $g_{\gamma(t)}$ is continuous on $\Lambda_{\gamma(t)}$ for each $t\in(-\ell,\ell)$.
We say that the family $\{g_{\sigma(t)}(\xi)\}_{-\ell<t<\ell}$ satisfies $Condition\ (RC)$ on $\sigma$ if
for any two points $q_1,q_2\in \sigma$ and  any sequence $\{\theta_j\}_{j=1}^\infty$ with $\theta_j\to 0$ as $j\to\infty$, there exists an isometry $T: \Lambda_{q_1}\to \Lambda_{q_2}$ and a subsequence $\{\delta_j\}$ of $\{\theta_j\}$ such that
\begin{equation}\begin{split} |\exp&_{q_1}(\delta_j l_1T\xi),\ \exp_{q_2}(\delta_j l_2\xi)|\\ \ls  &|q_1q_2|+\Big(\frac{(l_1-l_2)^2}{2|q_1q_2|}
-\frac{g_{q_1}(\xi)\cdot|q_1q_2|}{6}\cdot(l_1^2+l_1\cdot l_2+l^2_2)\Big)\cdot\delta^2_j+o(\delta_j^2)\end{split}\end{equation} for any $l_1,l_2\gs0$ and any $\xi\in \Lambda_{q_1}$.
\end{defn}

Clearly, the above Proposition 2.4 shows that the family $\{g_{\sigma(t)}(\xi)=k_0\}_{-\ell<t<\ell}$ satisfies Condition $(RC)$ on $\sigma$. In particular, if a family $\{g_{\sigma(t)}(\xi)\}_{-\ell<t<\ell}$ satisfies Condition $(RC)$, then the family  $\{g_{\sigma(t)}(\xi)\vee k_0\}_{-\ell<t<\ell}$ satisfies Condition $(RC)$ too.
\begin{defn} Let  $\gamma:[0,a)\to M$ be a geodesic. We say that $M$ has \emph{Ricci curvature bounded below by $ (n-1)K$ along $\gamma$},
if for any $\epsilon>0$ and any $0<t_0<a$, there exists $\ell=\ell(t_0,\epsilon)>0$ and a family of continuous functions $\{g_{\gamma(t)}(\xi)\}_{t_0-\ell<t<t_0+\ell}$ on $\Lambda_{\gamma(t)}$ such that the family satisfies $Condition\ (RC)$ on $\gamma|_{(t_0-\ell,\ t_0+\ell)}$ and \begin{equation}\oint_{\Lambda_{\gamma(t)}}g_{\gamma(t)}(\xi)\gs K-\epsilon \qquad \forall t\in(t_0-\ell,t_0+\ell),\end{equation}
where $\oint_{\Lambda_{x}}g_x(\xi)=\frac{1}{vol(\Lambda_x)}\int_{\Lambda_x}g_x(\xi)d\xi$.

 We say that $M$ has \emph{Ricci curvature  bounded below by $(n-1)K$} (\emph{locally}), denoted by $Ric(M)\gs (n-1)K$, if  each point $x\in M$ has  a neighborhood $U_x$   such that  $M$ has Ricci curvature bounded below by $(n-1)K$ along every geodesic $\gamma$ in $ U_x$.
 \end{defn}
\be{rem}{
(i) In a Riemannian manifold, this definition on Ricci curvature bounded below by $(n-1)K$ is exactly the  classical one. \\
\indent (ii) Let $M$ be an $n-$dimensional Alexandrov space  with curvature $\gs K$. The  Proposition 2.4 above shows that  $Ric(M)\gs(n-1)K$.}
\be{prop}
{ This curvature condition is a local condition. Namely, $Ric(M)\gs (n-1)K$ implies that $M$ has Ricci curvature  bounded below by $(n-1)K$ along  every geodesic in $M$.}
 Indeed,
 let $\gamma:[0,a)\to M$ be a geodesic. If $M$ has Ricci curvature bounded below by $ (n-1)K$ along $\gamma|_{[0,b_2)}$ and
$\gamma|_{[b_1,a)}$ with $0<b_1<b_2<a$, then $M$ has Ricci curvature bounded below by $ (n-1)K$ along $\gamma$.

Recently, in \cite{Pet-lv}, Petrunin proved that an $n-$dimensional Alexandrov space $M$ with  curvature $\gs K$ satisfies the curvature-dimension condition $CD(n,(n-1)K)$. Later in \cite{ZZ}, we can modify Petrunin's proof to prove the following
\be{prop}{$($see \cite{ZZ}$)$\indent An $n-$dimensional Alexandrov space $M$ having $Ric(M)\gs (n-1)K$  must satisfy $CD(n,(n-1)K)$.}

The relations among these various definitions on lower bound of Ricci curvature is summarized as follows:\\
\indent on an $n-$dimensional Alexandrov space $M^n$, there holds
$$ Ric\gs (n-1)K\Rightarrow CD(n,(n-1)K)\Rightarrow MCP(n,(n-1)K)\Leftrightarrow BG((n-1)K).$$
Obviously, all of these conditions are equivalent to each other on a smooth Riemannian manifold.
\begin{prob} Is the Ricci curvature condition $Ric(M)\gs(n-1)K$  equivalent to the curvature-dimension condition $CD(n,(n-1)K)$ on any $n-$dimensional Alexandrov spaces $M$? \end{prob}

\section{Basic comparison estimates}
\subsection{Laplacian comparison theorem}  Let $M$ be an $n-$dimensional Alexandrov space without boundary.  A canonical Dirichlet form $\mathcal{E}$ is defined by $$\mathcal{E}(u,v):=\int_M\ip{\nabla u}{\nabla v}d{\rm vol},\qquad {\rm for}\ \ u,v\in W^{1,2}_0(M).$$(see \cite{KMS}). The Laplacian associated to the canonical Dirichlet  form is given as follows.
Let $u: U\subset M\to \R$ be a $\lambda-$concave function. The (canonical) Lapliacian of $u$ as a sign-Radon measure is defined by
 \begin{equation}\int_M \phi d\Delta u=-\mathcal{E}(u,\phi)=-\int_M\ip{\nabla \phi}{\nabla u}d{\rm vol}\end{equation} for all Lipschitz function $\phi$ with compact support in $U.$ In \cite{Pet-lv}, Petrunin  proved \begin{equation*}\Delta u\ls n\lambda\cdot{\rm vol},\end{equation*} in particular, the singular part of $\Delta u$ is non-positive. If $M$ has curvature $\gs K$, then any distance function $dist_p(x):=d(p,x)$ is $\cot_K\circ dist_p-$concave on $M\backslash\{p\}$,  where the function $\cot_K(s)$ is defined  by $$\cot_K(s)=\begin{cases}\frac{\sqrt K\cdot\cos(\sqrt K s)}{\sin(\sqrt Ks)}\ & {\rm if}\quad K>0,\\
\frac{1}{s}\ &{\rm if}\quad K=0,\\
\frac{\sqrt{-K}\cdot\cosh(\sqrt{-K}s)}{\sinh(\sqrt{-K}s)}\ & {\rm if}\quad K<0.
\end{cases} $$Therefore the inequality $\Delta u\ls n\lambda\cdot{\rm vol}$ gives a Laplacian comparison theorem for the distance function on Alexandrov spaces.

   In \cite{KS-lap}, by using the $DC-$structure (see \cite{Per-dc}),  Kuwae-Shioya
defined a distributional Laplacian for a distance function $dist_p$ by $$\Delta dist_p=D_i\big(\sqrt{det(g_{ij})}g^{ij}\partial_j dist_p\big)$$ on a local chart of $M\backslash \mathcal S_\epsilon$ for sufficiently small  positive number $\epsilon$, where $$\mathcal S_\epsilon:=\{x\in M:\ {\rm vol}(\Sigma_x)\ls {\rm vol}(\mathbb{S}^{n-1})-\epsilon\}$$and $D_i$ is the distributional derivative. Note that the union of all $\mathcal S_\epsilon$ has zero measure. One can view the distributional Laplacian $\Delta dist_p$ as a sign-Radon measure. In \cite{KMS}, Kuwae, Machigashira and Shioya proved that  the distributional Laplacian is  actually a representation of  the previous (canonical) Laplacian in $M\backslash \mathcal S_\epsilon$.  Moreover in \cite{KS-lap}, Kuwae and Shioya extended the Laplacian comparison theorem  under
the weaker condition $BG(k)$:
\begin{thm} $($\cite{KS-lap}$)$\indent If an $n-$dimensional Alexandrov space $ M$ satisfies $BG(k)$,
then $$ d \Delta dist_p \ls (n -1) \cot_k\circ dist_p\cdot d{\rm vol}\qquad {\rm on}\quad M\backslash(\{p\}\cup S_\epsilon).$$
\end{thm}

 Both of the above canonical Laplacian  and its DC representation (i.e. the distributional Laplacian) make sense up to a set
which has zero measure.  In particular, they do not make sense along
a geodesic.

In \cite{ZZ}, the authors  defined a new version of Laplacian for a
distance function from a given point $p\in M$ along a geodesic and
proved a comparison theorem for the new defined Laplacian under the
Ricci curvature condition defined in Section 2.3. This version of
Laplacian comparison theorem makes pointwise sense in $W_p:=
M\backslash C_p$.

  Fix a geodesic $\gamma:[0,\ell)\rightarrow M$ with  $\gamma(0)=p$ and denote $f:=dist_{p}$. Let $x\in \gamma\backslash\{p\}$ and $L_x,\ \Lambda_x$ be as above in the end of Section 1. Clearly, we  may assume  that $M$ has curvature $\gs k_0$ (for some $k_0<0$)  in a neighborhood  $U_\gamma$ of $\gamma$.

Throughout this paper,  $\mathcal S$ will always denote the set of
all sequences $\{\theta_j\}^\infty_{j=1}$ with $\theta_j \to0$ as
$j\to\infty$ and $\theta_{j+1}\ls \theta_j$.

 We define a  version of Hessian and Laplacian for the distance function $f$ along the geodesic $\gamma$ as follows.
\begin{defn}Let  $x\in \gamma\backslash\{p\}$.  Given a sequence $\theta:=\{\theta_j\}_{j=1}^\infty\in\mathcal S$, we define a function $H^\theta_xf:\Lambda_x\rightarrow\mathbb{R}$ by
$$H^\theta_xf(\xi)\overset{def}{=}\limsup_{s\to0, \ s\in \theta}\frac{f\circ\exp_x(s\cdot\xi)-f(x)}{s^2/2};$$
and $$\Delta^\theta f(x)\overset{def}{=}(n-1)\cdot\oint_{\Lambda_x} H^\theta_xf(\xi).$$
\end{defn}

Denote by $Reg_h$ the regular set of semi-concave function $h$, (i.e., the set of points $z\in M$ such that $z$ is regular and $Hess_zh$ is a bilinear  form on $T_z$).
If we write  the Lebesgue decomposition of the canonical Laplacian $\Delta f=(\Delta f)^{sing}+(\Delta f)^{ac}\cdot{\rm vol}$,  with respect to the $n-$dimension Hausdorff measure ${\rm vol}$, then $(\Delta f)^{ac}(x)={\rm Tr}Hess_xf=\Delta^\theta f(x)$ for all $x\in W_p\cap Reg_f$ and $\theta\in\mathcal S$. It was shown  in \cite{OS-rie,Per-dc} that $Reg_f\cap W_p$ has full measure. Thus  $\Delta^\theta f(x)$ is actually a representation of the absolutely continuous part of the canonical Laplacian $\Delta f$ in $W_p$.

The propagation of the above defined Hessian along the geodesic $\gamma$ is described by the following result.
\begin{prop}  $($\cite{ZZ}$)$ Let $\{g(\xi)\}_{t_0-\epsilon<t<t_0+\epsilon}$ be a family of continuous functions on $\Lambda_{\gamma(t)}$ which satisfies Condition $(RC)$  on $\gamma|_{(t_0-\epsilon,t_0+\epsilon)}.$  Consider a sequence $\{\theta_j\}_{j=1}^\infty\in \mathcal S$, and $y,z\in \gamma|_{(t_0-\epsilon,t_0+\epsilon)}$ with $|py|<|pz|.$ Assume that a isometry $T:\Lambda_z\to\Lambda_y$ and  a subsequence $\delta:=\{\delta_j\}\subset\{\theta_j\} $ such that (2.8) holds. Then  we have \begin{equation}H^\delta_zf(\xi)\leqslant l^2\cdot H^\delta_yf(\eta)+ \frac{(l-1)^2}{|yz|}-\frac{l^2+l+1}{3}\cdot|yz|\cdot g_z(\xi) \end{equation}for any $l\geqslant0$ and any $\xi\in\Lambda_z,\ \eta=T\xi\in\Lambda_y$.
\end{prop}
By using the above propagation inequality for the Hessian, we can obtain the following comparison for the new defined Laplacian.
\begin{thm}$($\cite{ZZ}$)$  Let  $x\in \gamma\backslash\{p\}$. If  $M$ has Ric $\geqslant (n-1)K$ along the geodesic $\gamma(t)$, then,
given  any sequence $\{\theta_j\}_{j=1}^\infty\in\mathcal S$, there exists a subsequence $\delta=\{\delta_j\}$ of $\{\theta_j\} $
 such that $$\Delta^{\delta} f(x)\ls(n-1)\cdot\cot_K(|px|).$$
(If $K>0$, we add assumption $|px|<\pi/\sqrt K$).

Furthermore,  if
\begin{equation}\Delta^{\theta'} f(x)= (n-1)\cdot\cot_K(|px|)\end{equation}for any subsequence $\theta'=\{\theta'_j\}$ of $\delta$, then
there exists a subsequence $\delta'=\{\delta'_j\}$ of $\delta$ such that \begin{equation}H^{ \delta'}_xf= \cot_K(|px|)\end{equation} almost everywhere in $\Lambda_x$.
\end{thm}
\begin{proof}A sketch of the proof is given as follows.
Fixed any number $\epsilon>0$ and $K'<K$, we can choose $y\in\gamma$ between $p$ and $x$ such that $$\cot_{k_0}(|py|)\ls \cot_{K'}(|py|-\epsilon).$$
Divide the segment $\gamma_{yx}$ sufficiently fine by points $x_0=y,\ x_1,\ \cdots, \ x_{N-1}$ and $x_N=x$ with  $|px_j|<|px_{j+1}|.$ By using the lower bound of Ricci curvature, Proposition 3.3 and  an induction argument, we can prove that $$\Delta^{\delta} f(x_j)\ls(n-1)\cdot\cot_{K'}(|px_j|-\epsilon)$$ for all $1\ls j\ls N$ and some subsequence $\delta$. Then a standard diagonal argument will imply  the first assertion of the theorem.

Now we suppose
\begin{equation*}\Delta^{\theta'} f(x)= (n-1)\cdot\cot_K(|px|)\end{equation*}for any subsequence $\theta'=\{\theta'_j\}$ of $\delta$. From a discrete argument, we can get that, for any $\epsilon>0$, there is  a subsequence
$\delta'=\{\delta'_j\}$ of $\delta$ and an integrable function $h$ on
$\Lambda_x$ such that $$H^{\delta'}_xf\ls h\quad {\rm and }\quad
\oint_{\Lambda_x} \big(h-\cot_K(|px|)\big)^2\ls
\big(3+2|\cot_K(|px|)|\big)\epsilon.$$
By a standard diagonal argument, we can obtain a new subsequence of $\delta$, denoted by  $\delta'$ again, such that $$ H^{\delta'}_xf\ls \cot_K(|px|)$$ almost everywhere in $\Lambda_x.$  This implies  the second assertion of the theorem.
\end{proof}
\be{rem}{Consider the canonical Laplacian $\Delta f$ which is a sign-Radon measure. Its Lebesgue decomposition  with respect to the $n-$dimension Hausdorff measure ${\rm vol}$ is written as $\Delta f=(\Delta f)^{sing}+(\Delta f)^{ac}\cdot{\rm vol}$. The above Theorem 3.4 gives an upper bound for the continuity part $(\Delta f)^{ac}$. We have seen  that $ (\Delta f)^{sing}$ is non-positive. Thus Theorem 3.4 is actually  giving a  Laplacian comparison theorem in sense of measure (or distribution).}

We now define the \emph{upper Hessian} of $f$, $\overline{Hess}_xf:\ T_x\rightarrow\mathbb{R}\cup\{-\infty\}$ by  \begin{equation}\overline{Hess}_xf(v,v)\overset{def}{=}\limsup_{s\to0}\frac{f\circ\exp_x(s\cdot v)-f(x)-d_xf(v)\cdot s}{s^2/2}\end{equation} for any $v\in T_x$.
Clearly, this definition also works for any semi-concave function on $M$. Indeed, if $u$ is a $\lambda-$concave function, then its upper Hessian $\overline{Hess}_xu(\xi,\xi)\ls \lambda$ for any $\xi\in\Sigma_x$.

Given $K\in\R$, consider the function $\varrho_K$ defined by$$\varrho_K(\upsilon)=\begin{cases}\frac{1}{K}\big(1-\cos(\sqrt K\upsilon)\big)\ & {\rm if}\quad K>0,\\
\frac{\upsilon^2}{2}\ &{\rm if}\quad K=0,\\
\frac{1}{K}\big(\cosh(\sqrt{-K}\upsilon)-1\big) \ & {\rm if}\quad
K<0.
\end{cases} $$
The following proposition is concerned with  the rigidity part of Theorem 3.4.
\begin{prop}$($\cite{ZZ}$)$
Let $M$ be an $n-$dimensional Alexandrov space with $Ric(M) \gs
(n-1)K$.  Suppose that $B_p(R)\backslash\{p\}\subset W_p$ for some
$0<R\ls\pi/\sqrt K$ (if $K\ls 0$, we set $\pi/\sqrt K$ to be
$+\infty$). Assume that for almost every $x\in
B_p(R)\backslash\{p\}$, there exists a sequence $\theta:=
 \{\theta_j\}_{j=1}^\infty\in \mathcal S$ such that $\Delta^{\theta'}f(x)=\cot_K(|px|)$ for any subsequence $\theta'\subset \theta$.

  Then the function $\varrho_K\circ f$ is $(1-K\cdot\varrho_K\circ f)-$concave in $B_p(R)\backslash\{p\}$.

Consequently, if $\sigma(t)$ and $\varsigma(t)$ are  two geodesics in $B_p(R)$ with $\sigma(0)=\varsigma(0)=p$, and   $$\varphi(\tau,\tau')=\wa_K \sigma(\tau)p\varsigma(\tau')$$ is the comparison angle of $\angle \sigma(\tau)p\varsigma(\tau')$ in the $K-$plane, then  $\varphi(\tau,\tau')$ is non-increasing with respect to $\tau$ and $\tau'$.

(If $K>0$, we add the assumption that
$\tau+\tau'+|\sigma(\tau)\varsigma(\tau')|<2\pi/\sqrt K$).
 \end{prop}
\begin{proof}This proposition was proved in \cite{ZZ}. We now describe the ideas of its proof. It suffices  to  show that one variable function  $h_p:=\varrho_K\circ f\circ\gamma(s)$ satisfies that
$$h_p''\ls 1-Kh_p$$ for any geodesic $\gamma(s)\subset B_p(R)\backslash\{p\}.$

We consider the function $u: W_p\rightarrow \mathbb{R}^+\cup\{0\}$, \begin{equation}u(z)=\sup_{\xi\in \Sigma_z}\Big|\overline{Hess}_zf(\xi,\xi)-\cot_K(|pz|)\cdot\sin^2(|\xi,\uparrow^p_z|)\Big|.\end{equation}

By the assumption,  we have $u(z)=0,$ almost everywhere in $B_p(R)\cap Reg_f$.

  Since $Reg_f$ has full measure in $B_p(R),$  we conclude that $u\equiv0$ almost everywhere in $B_p(R)$.
By Cheeger-Colding's segment inequality (see \cite{CC-rigi} or Section 5), we can choose geodesic $\gamma_{x_1,y_1}$ such that it closes to $\gamma_{xy}$ arbitrarily and  $u(\gamma_{x_1,y_1}(s))=0$ almost everywhere on $(0,|x_1y_1|)$.
 For the function $\wt f(s)=f\circ \gamma_{x_1,y_1}(s)$, we get
 $$\wt f''(s)\ls \cot_K\wt f(s)\cdot\big(1-\wt{f'}^2(s)\big)$$ for almost everywhere $s\in (0,|x_1y_1|)$.
Thus the function $\wt h(s)=\varrho_K\circ\wt f(s)$ satisfies $$\wt h''(s)\ls 1-K\wt h(s)$$
for almost everywhere $s\in (0,|x_1y_1|)$. Then, by letting $\wt h\to h$, this will prove the desired inequality.
\end{proof}

\subsection{Bishop-Gromov volume comparison theorem}
Let $(X,d,m)$ be a metric measure space. It is proved that  $CD(n,(n-1)K)$ or $MCP(n,(n-1)K)$ implies  Bishop-Gromov relative volume comparison theorem, (see \cite{S-acta,LV-ann,O-mcp}).  In particular, Bishop-Gromov relative volume comparison is also held on an $n-$dimensional Alexandorv space with Ricci curvature $\gs (n-1)K$. That is, the function $$\frac{{\rm vol}B_p(R)}{{\rm vol}\wt B_o(R)}$$ is non-increasing with respect to $R>0$, where $\wt B_o(R)$ is a geodesic ball with radius $R$ in $M^n_{K}$.
The rigidity part  is  discussed in the following result.
\begin{thm} Let $ M$ be an $n-$dimensional Alexandrov space without boundary and $p\in M$. Assume $Ric(M) \gs (n-1)K$ and suppose
 \be{equation*}{{\rm vol}B_p(R)={\rm vol}\wt B_o(R)} for some $R>0.$ Then $B_p(R/2)$ is isometric to $\wt B_o(R/2)$.
\end{thm}
\begin{proof}By the assumption, we have  \begin{equation}{\rm vol}\partial B_p(r)={\rm vol}\partial \wt B_o(r)\end{equation} for all $0<r<R$.

We claim that $B_p(r)\backslash\{p\}\subset W_p$ for all $0<r<R.$ Let us argue by contradiction.

 Suppose that there exists a point $q\in C_p$ with $|pq|=r_1<R$. Then we can find a neighborhood $U_q\ni q$ such that for any point $x\in U_q$, geodesic $\gamma_{px}$ does not extend beyond $x$ with length $\frac{R-r_1}{2}.$ Now take $r_2\in (\frac{R+r_1}{2},R)$ and set
$$A(r_1)=\{y\in \partial B_p(r_1):\ \exists z\in \partial B_p(r_2)\ {\rm such\ that}\ |pz|=|py|+|yz| \},$$ then $U_q\cap A(r_1)=\varnothing.$
By condition $BG((n-1)K)$, (this is implied by $Ric\gs(n-1)K$), we have
$$\frac{{\rm vol}\partial B_p(r_2)}{{\rm vol}\partial\wt B_o(r_2)}\ls \frac{{\rm vol}A(r_1)}{{\rm vol}\partial \wt B_o(r_1)}< \frac{{\rm vol}\partial B_p(r_1)}{{\rm vol}\partial \wt B_o(r_1)}.$$
This contradicts to equation (3.7).

 We now consider $\Delta dist_p$ as a sign-Radon measure. Fix any two numbers $a,b\in (0,R)$ and a nonnegative Lipschitz function $\phi:\R\to\R$ with support in $[a,b]$. By applying co-area formula, we have
 \begin{equation}\int_M \phi(|px|)d\Delta dist_p=-\int_M\ip{\nabla\phi}{\nabla dist_p}d{\rm vol}=-\int_a^b\phi'(r)\cdot{\rm vol}\partial B_p(r)dr.\end{equation}
Laplacian comparison implies
 \begin{equation}\int_M \phi(|px|)d\Delta dist_p\ls\int_M\phi(|px|)\cot_K(|px|)d{\rm vol}=\int_a^b\phi(r)\cot_K(r)\cdot{\rm vol}\partial B_p(r)dr.\end{equation}
 By using (3.7), (3.8) and co-area formula in $M^n_K$, we have $$\int_M \phi(|px|)d\Delta dist_p=\int_M\phi(|px|)\cot_K(|px|)d{\rm vol}.$$
 The combination of this, (3.9) and the arbitrariness of $a,b$ shows that for almost everywhere $x\in Reg_{dist_p}\cap B_p(R)$,  we have $$\Delta^\theta dist_p(x)=(n-1)\cdot\cot_K(|px|)$$ for all $\theta\in \mathcal S.$

Therefore, we can apply Proposition 3.6 to conclude  that there exists an expanding map $F$ from $B_p(R)$ to  $\wt B_o(R)\subset M^n_K$ with $F(p)=o$ and $|F(x)F(p)|=|px|$. Now we can show that it is an isometry from $B_p(R/2)$ to $\wt B_o(R/2)$.

Suppose that there are two points $x,y\in B_p(R/2)$ such that $|xy|<|F(x)F(y)|$. We set $\bar x=F(x)$, $\bar y=F(y)$ and $|xy|=2s$, $|\bar x\bar y|=2\bar s< R.$ Since $F$ is expanding, we have $$F\Big(\big(B_x(\bar s)\cup B_y(\bar s)\big)^c\Big)\subset \Big(\wt B_{\bar x}(\bar s)\cup \wt B_{\bar y}(\bar s)\Big)^c.$$
Then it follows from the  assumption ${\rm vol}B_p(R)={\rm vol}\wt B_o(R)$  that $${\rm vol}\big(B_x(\bar s)\cup B_y(\bar s)\big)\gs {\rm vol}\big(\wt B_{\bar x}(\bar s)\cup\wt  B_{\bar y}(\bar s)\big)=2{\rm vol}\wt B_o(\bar s).$$

On the other hand, letting $z$ be a mid-point of $x$ and $y$, we have
 $${\rm vol}\big(B_x(\bar s)\cup B_y(\bar s)\big)+{\rm vol}B_z(\frac{\bar s-s}{2})\ls {\rm vol}B_x(\bar s)+{\rm vol}B_y(\bar s)\ls 2{\rm vol}\wt B_o(\bar s).$$
This contradicts to $s<\bar s$.
\end{proof}

The next result extends Abresch-Gromoll's excess estimate from Riemannian manifolds to Alexandrov spaces. Let $M$ be an Alexandrov space without boundary. For $q_+, q_-\in M$, the excess function $E$ with respect to $q_+$ and $q_-$ is $$E(x)=|xq_+|+|xq_-|-|q_+q_-|.$$
\be{prop}
{If $M$ satisfies $BG(-\eta)$ $(\eta\gs0)$ and for $x\in M$ with $|xq_+|+|xq_-|\gs L$, $E(x)\ls \epsilon$, then on $B_x(R)$ we have $$E\ls \Phi(\eta,L^{-1},\epsilon\ |\ n,R),$$
where $ \Phi(\eta,L^{-1},\epsilon|n,R)$ is a positive function such that for fixed $n$ and $R$, $\Phi$ tends to zero as $\eta,\epsilon\to0$ and $L\to\infty$.}
One can check that the same proof in Riemannian manifolds (see for example \cite{Cheeger}) also works for Alexandrov spaces.

\section{Geometric consequences}
In this section, we summarize  geometric consequences for Alexandrov spaces under the generalized  Ricci condition.

Let $(X,d,m)$ be  a metric measure space satisfying $CD(n,k)$,   Sturm and Lott-Villani have proved  the following  geometric results:
Brunn-Minkowski inequlity \cite{S-acta}; Bishop-Gromov volume comparison \cite{S-acta,LV-ann}; Bonnet-Myers estimate on diameter \cite{S-acta,LV-ann} and Lichnerowicz estimate on the first eigenvalue of the Laplacian \cite{LV-jfa}.

If $(X,d,m)$  satisfies $MCP(n,k)$, Ohta in \cite{O-mcp} and Sturm in \cite{S-acta} proved Bishop-Gromov volume comparison theorem and Bonnet-Myers  theorem on diameter.

It is obvious that  all of these results also hold  for Alexandrov spaces with the Ricci lower bound condition defined in Section 2.3.
\subsection{Rigidity theorems}
The simplest rigidity is that a smooth $n-$dimensional  Riemannian manifold $M$ with $Ric\gs n-1$ and ${\rm vol}(M)={\rm vol}(\mathbb S^n)$ must be isometric to $\mathbb S^n$. Cheng's maximal diameter theorem asserts that the rigidity still holds  under the assumption ${\rm diam}(M)=\pi={\rm diam}(\mathbb S^n)$ and $Ric\gs n-1$.

Perhaps, the most important rigidity is Cheeger-Gromoll splitting theorem. It states that every Riemannian manifold with nonnegative Ricci curvature and containing a line must split out a factor $\R$ isometrically.

 For Alexandrov spaces, the following topological rigidities results were proved by Ohta \cite{O-product} and Kuwae-Shioya in \cite{KS-lap}.
\begin{thm}Let $M$ be an $n-$dimensional Alexandrov spaces without boundary.\\ $(1) ($Ohta \cite{O-product} and  Kuwae-Shioya  \cite{KS-lap}$)$\indent If $M$ satisfies $BG(n-1)$ and ${\rm diam}(M)=\pi$, then it is homeomorphic to a suspension.\\
$(2)($Kuwae-Shioya  \cite{KS-lap}$)$\indent If $M$ satisfies $BG(0)$ and contains a line, then $M$ is homeomorphic to a direct product
space $N\times \mathbb{R}$ over some topological space $N$.
\end{thm}
Actually, in \cite{KS-wei}, Kuwae-Shioya  obtained a more general
weighted measure version of the second assertion in the above theorem.

 In \cite{ZZ}, under the corresponding Ricci curvature conditions, the authors obtained the following metric rigidity results:
\begin{thm} $($Zhang-Zhu \cite{ZZ}$)$\indent
 Let $M$ be  an $n-$dimensional Alexandrov space without boundary.\\
 \indent$(i)$\indent {\rm (Maximal diamter theorem)} If $Ric(M)\gs n-1$  and  $diam(M)=\pi$ , then $M$ is isometric to  suspension $[0,\pi]\times_{\sin}N,$ where $N$ is an Alexandrov space with curvature $\gs1.$\\
 \indent$(ii)$\indent {\rm (Splitting theorem)} If $M$ has nonnegative Ricci curvature and contains a line, then $M$ is isometric to  direct metric product $\R\times N,$ where $N$ is an Alexandrov space with  nonnegative Ricci curvature.
\end{thm}
In \cite{CC-rigi,CC-ric}, Cheeger and Colding studied  the limiting spaces of  smooth Riemannian manifolds under Gromov-Hausdorff convergence. Among other things in \cite{CC-rigi}, they extended Cheng's maximal diameter theorem and Cheeger-Gromoll's splitting theorem to the limiting spaces. One of challenge problem in Alexandrov space theory is whether any Alexandrov space can be approximated by smooth Riemannian manifolds via Gromov-Hausdorff topology\footnote{In a private conversation with Y. Burago, the second author learnt that the experts in the field guess the negative answer to the problem.}. The above rigidity theorem might shed the light to answer this challenge problem.

A sketch of the proof of the above metric rigidity theorem is given as follows.

To illustrate the proof for the maximal diameter theorem, let us
take two points $p,q\in M$ with $|pq|=\pi.$ One can check $${\rm vol}B_p(r)+{\rm vol}B_q(\pi-r)={\rm vol}(M)$$for all $0<r<\pi.$ Further, by using Bishop-Gromov volume comparison theorem, we have $$\frac{{\rm vol}B_p(r)
}{{\rm vol}\wt B(r)}=\frac{{\rm vol}(M)}{{\rm vol}(\mathbb S^n)}$$ for all $0<r<\pi$, where $\wt B(r)\subset \mathbb S^n$ is a geodesic ball with radius $r$.

Set $f=dist_p$ and $\bar f=dist_q$. By applying   an argument similar in Proposition 3.6, we show that $-\cos f$ is $\cos f-$concave  and  $-\cos \bar f$ is $\cos \bar f-$concave in $W_q=W_p=M\backslash\{p,q\}$.
Thus by combining  $\cos f=-\cos\bar f$,  we get \begin{equation} (-\cos f\circ\sigma)''(s)=
\cos f\circ\sigma(s)
\end{equation} for any geodesic $\sigma\in M\backslash\{p,q\}$.

Denote by $$M^+=\big\{x\in M:\ f(x)\ls \pi/2\big\},\qquad  M^-=\big\{x\in M:\ f(x)\gs \pi/2\big\}$$
and $N=M^+\cap M^-=\{x\in M:\ f(x)= \pi/2\}.$

Using Proposition 3.6 and the equation (4.1), by a direct calculation, we can show that $M$ is isometric to suspension $[0,\pi]\times_{\sin}N$. Moreover, $N$ is isometric to $\Sigma_p$, the space of directions at $p$.

\begin{cor}
 Let $M$ be  an $n-$dimensional Alexandrov space with $Ric(M)\gs n-1$. If $rad(M)=\pi$, then $M$ is isometric to the sphere $\mathbb S^n$ with the standard metric.
\end{cor}
\begin{proof} For any point $p\in M$, there exists a point $q$ such that $|pq|=\pi$. From the proof of Maximal diameter theorem, we have that $-\cos dist_p$ is $\ \cos dist_p-$concave in $B_p(\pi)\backslash\{p\}$. It follows from the arbitrariness of $p$  that $M$ has curvature $\gs1$. It is well-known  that  an $n-$dimensional Alexandrov space with curvature $\gs1$ and $rad=\pi$ must be isometric to  the sphere $\mathbb S^n$ with the standard metric.
\end{proof}

 The well known Obata theorem asserts that if the first eigenvalue $\lambda_1$ of the Laplacian on an $n-$dimensional Riemannian manifold with $Ric\gs n-1$ is equal to $n$, then the diameter of the Riemannian manifold is $\pi$, and hence it is a standard sphere.

  In \cite{LV-jfa}, Lott-Villani proved that a metric measure space $X$  has $\lambda_1(X)\gs nk/(n-1)$ provided $M$ satisfies $CD(n,k)$. Hence,  an $n-$dimensional  Alexandrov space $M$ with $Ric\gs n-1$ must satisfy $\lambda_1(M)\gs n$. Note that the spherical suspensions over  Alexandrov spaces with curvature $\gs 1$ satisfy $\lambda_1(M)= n$. So a problem arises whether Obata's theorem holds true or not for Alexandrov spaces. Precisely,
  \begin{prob} Let $M$ be an $n-$dimensional  Alexandrov space (without boundary) with  $Ric(M)\gs n-1$ and  $\lambda_1(M)= n$. Is its diamter equal to $\pi$?\end{prob}If the answer is ``yes'', then by the above Maximal diameter theorem,  $M$ must be a suspension.
 \begin{prob} Can one prove Levy-Gromov  isoperimetric inequality for Alexandrov spaces under the Ricci curvature condition $Ric(M)\gs n-1$. Precisely:

  Let $M$ be an $n-$dimensional  Alexandrov space (without boundary) with  $Ric(M)\gs n-1$. Set  a surface  $\sigma_\alpha$
which divides the volume of $M$ in ratio $\alpha$. Let $s_\alpha$ be a geodesic
sphere in $\mathbb{S}^n$ which divides the volume of $\mathbb{S}^n$ in the same ratio $\alpha$. Can one prove $$\frac{{\rm vol}(\sigma_\alpha)}{{\rm vol}(M)}\gs \frac{{\rm vol}(s_\alpha)}{{\rm vol}(\mathbb{S}^n)}?$$\end{prob}
In \cite{Pet-har}, Petrunin sketched a proof to  Levy-Gromov isoperimetric inequality  for Alexandrov spaces with curvature $\gs1$.\\

To state the idea of the proof for the above splitting theorem, let us review what is the proof  in smooth case.

 Let $M$ be a smooth Riemannian manifold with nonnegative Ricci curvature and fix a line $\gamma(t)$ in $M$.  We set $\gamma_+=\gamma|_{[0,+\infty)}$, $\gamma_-=\gamma|_{(-\infty,0]}$. Let $b_+$ and $b_-$ be the Busemann functions for rays $\gamma_+$ and $\gamma_-$, respectively.

 By Laplacian comparison theorem, $b_+$ and $b_-$ are subharmonic on $M$. It follows from  the maximum principle that $b_++b_-=0$ on $M$. Thus thery are harmonic. Elliptic equation regularity theory tells us that they are smooth. The important step is to ues  Bochner formula to conclude that both $\nabla b_+$ and $\nabla b_-$ are parallel. Consequently, the splitting theorem follows directly from de Rham decomposition theorem.

  If an $n-$dimensional Alexandrov spaces $M$ satisfies $BG(0)$, then the distribution Laplacian comparison theorem and the maximum principle  still hold. Kuwae-Shioya \cite{KS-lap} proved that the Busemann functions  $b_+(x)$ and $b_-(x)$ are harmonic, when $M$ contains a line and satisfies $BG(0)$. The main difficulty is that neither smoothness of harmonic functions nor Bochner formula is available in Alexandrov spaces. Indeed, one does not expect the Bochner formula holds on Alexandrov spaces.

In our proof for the splitting theorem, the key step is to prove $b_+$ is an affine function, i.e., $b_+\circ\sigma(t)$ is linear for any geodesic $\sigma(t)$ in $M$ (in smooth case, this is $\nabla\nabla b_+=0$). Firstly, we prove  that the Busemann functions $b_+$ and $b_-$ are semi-concave. This fact allows us to define a pointwise Laplacian for them. Next we prove  that $b_+$ and $b_-$ are concave. Then the combination of the concavity of $b_+, b_-$ and the fact that $b_+(x)+b_-(x)=0$ imply that  $b_+$ is an affine function  $M$.

 Finally, we adapt an argument of Mashiko \cite{Mashi} to prove that $M$ is isometric to a direct product $\mathbb{R}\times N$ over an Alexandrov space $N$. Furthermore, we can prove that $N$ has nonnegative Ricci curvature.

As an consequence of  the splitting theorem, we get a rigidity for $n-$dimensional torus:\be{cor}{ Any  Alexandrov metric on $\mathbb T^n$ with nonnegative Ricci curvature must be a flat metric.}
\begin{proof} Let $d$ be an Alexandrov metric on $\mathbb T^n$ with nonnegative Ricci curvature. The topological product space $M^n=\mathbb T^{n-1}\times\R$ covers $(\mathbb T^n,d)$, and hence has a lifted metric $\wt d$ on $M^n$. Note that $M^n$ has two ends. Thus, it contains a line. By the above splitting theorem, $\wt d$ must split isometrically out an Alexandrov metric $\widehat{d}$ on $\mathbb T^{n-1}$, it still has nonnegative Ricci curvature. Therefore, the desired result  follows by induction on  dimension.
\end{proof}
\subsection{Fundamental group}
Any small spherical neighborhood of a point in an  Alexandrov space is homeomorphic to an open cone \cite{Per-morse}, and hence it is locally contractible and exists a universal cover.
Since the lower  Ricci curvature bounds for Alexandrov spaces is local (c.f. Proposition 2.8), it implies that its universal cover has the same lower bound for the Ricci curvature.

 In \cite{BS}, Bacher-Sturm proved that ``local $CD(n,k)$'' implies ``globe $CD(n+1,k)$'' for metric measure spaces (see  \cite{BS}).  Hence one can get some estimates for fundamental group and Betti number on Alexandrov spaces under proper condition $CD(n,k)$, which are similar (but, weaker)
as Corollary 4.7 and 4.8  below.

 From the above splitting theorem, we can apply the same proofs as in Riemannian manifold case (see, for example, Section 3.5 in Chapter 9 of \cite{Petersen} and \cite{A})  to get the following structure theorem for fundamental group of  Alexandrov space with nonnegative Ricci curvature and a theorem of Milnor type.
\begin{cor}Let $M$ be a compact $n-$dimensial Alexandrov space with nonnegative Ricci curvature. $\partial M=\varnothing.$ Then its fundamental group has a finite index Bieberbach subgroup.
\end{cor}
\be{cor}{Let $M$ be an $n-$dimensial Alexandrov space with nonnegative Ricci curvature. $\partial M=\varnothing.$ Then any finitely generated subgroup of $\pi_1(M)$ has polynomial growth of degree $\ls n$.

Moreover, if some finitely generated subgroup of $\pi_1(M)$ has polynomial growth of degree $= n$, then $M$ is compact and flat.}

By using  Bishop-Gromov volume comparison theorem  on its universal cover, the same proof as in Riemannian manifold case (see, for example, page 275-276 in \cite{Petersen}) give the
 following estimates on  the first Betti number.
\begin{cor}
   Let $M$ be an $n-$dimensional  Alexandrov space with  $\partial M=\varnothing$. If $Ric(M)\gs(n-1)K$ and diameter of $M\ls D$, then   $$b_1(M)\ls C(n, K^2\cdot D)$$ for some function $C(n, K^2\cdot D)$.

    Moreover, there exists a constants $\kappa(n)>0$ such that if $K^2\cdot D\gs -\kappa(n),$ then $b_1(M)\ls n.$
\end{cor}

\section{Analytic consequences}
In this section, we summarize  basic analytic consequences on Alexandrov spaces, including Poincar\'e inequality, Sobolev inequality and Lipschitz continuity  of harmonic functions and heat kernel.  Most analytic properties are obtained under condition  $BG(k)$ (or $MCP(n,k)$).
\subsection{Poincar\'e, Sobolev and mean value inequality}

 Kuwae,  Machigashira and Shioya   proved a
Poincar\'e inequality with a constant depending on the volume (see Theorem 7.2 in \cite{KMS}).
A. Ranjbar-Motlagh in \cite{Ran} proved a
Poincar\'e inequality under a measure contraction property. Lott and Villani in \cite{LV-jfa} proved a Poincar\'e inequality for  metric measure spaces with a ``democratic'' condition (see \cite{LV-jfa} for the definition of the ``democratic'' condition).

Cheeger-Colding in \cite{CC-rigi} proved the following segment inequality for Riemannian manifolds with lower Ricci bounds, which was later extended to non-branching metric measure spaces satisfying $MCP(n,k)$ by M. von Renesse in \cite {Ren}. In particular, the  inequality holds  for Alexandrov spaces with condition $BG(k)$.
\begin{lem}{\bf (Cheeger-Colding's segment\ inequality)}
Let $M$ be an $n-$dimensional Alexandrov space satisfying $BG(k)$. Given a nonnegative measurable function $g$ on $M$, set $$\mathcal F_g(x,y)=\inf_\gamma\int_0^lg\circ\gamma(s)ds,$$ where the inf is taken over all minimal geodesics $\gamma$ from $x$ to $y$. Then for any two measurable sets $A_1$ and $A_2$ with $A_1,A_2\subset B_p(r)$, there holds \be{equation}{
\int_{A_1}\int_{A_2}\mathcal F_g(x,y)dxdy\ls c(n,R)\cdot r\cdot({\rm vol}(A_1)+{\rm vol}(A_2))\cdot\int_{B_p(2R)}g(z)dz,}where $$c(n,R)=2\sup_{0<s/2\ls u\ls s\ls R}\frac{{\rm vol}(\partial \wt B(s))}{{\rm vol(\partial} \wt B(u))}$$ and $\wt B(s)$ is a geodesic ball with radius $s$ in model space $M^n_{k/(n-1)}.$
\end{lem}
By combining this segment inequality with the following Cauchy-Schwarz inequality $$|f(y)-f(x)|^2\ls \Big(\int_0^{|xy|}|\nabla f|\circ\gamma(s)ds\Big)^2\ls |xy|\cdot\int_0^{|xy|}|\nabla f|^2\circ\gamma(s)ds,$$ one immediately gets a  (weaker) $L^2-$Poincar\'e inequality:
\be{prop}{Let $M$ be an $n-$dimensional Alexandrov space  satisfying $BG(k)$. Then for any $f\in W^{1,2}(M)$ we have
\be{equation}{\int_{B_p(R)}|f-f_B|^2\ls c(n,R)\cdot R^2\int_{B_p(2R)}|\nabla f|^2,}where $f_B=\frac{1}{{\rm vol}B_p(R)}\int_{B_p(R)}f$.
In particular, if $M$ satisfies $BG(0)$, then the constant $c(n,R)=2^{n}$.}

Furthermore, the combination of  double condition\footnote{A subset $\Omega\subset M$ is said to satisfy a double condition with double constant $D_\Omega,$ if  ${\rm vol}(B_x(2r))\ls D_\Omega\cdot{\rm vol}(B_x(r))$ for all $0<r<R$ and  $x\in \Omega$ with $B_x(2r)\subset\Omega$.} and the (weaker) Poincar\'e inequality implies  the following  Poincar\'e inequality and  Sobolev inequality.
\be{prop}{Let $M$ be an $n-$dimensional Alexandrov space  satisfying $BG(k)$. Then for any $f\in W^{1,2}(M)$ we have
\be{equation}{\int_{B_p(R)}|f-f_B|^2\ls c^*(n,R)\cdot R^2\int_{B_p(R)}|\nabla f|^2,}for a new constant $c^*(n,R),$ which is depending only on the above $c(n,R)$ and the double constant $D_{B_p(2R)}$.
}

\be{prop}{Let $M$ be an $n-$dimensional $(n\gs3)$ Alexandrov space  satisfying $BG(k)$.  Then exists a constant $c^*(n,R)$ such that for any $f\in W^{1,2}(M)\cap C_0(B_p(R))$, we have
\be{equation}{\big(\int_{B_p(R)}|f|^{2N}\big)^{1/N}\ls c^*(n,R)\cdot \frac{R^2}{({\rm vol}B_p(R))^{n/2}}\int_{B_p(R)}(|\nabla f|^2+R^{-2}\cdot f^2),} where
$N=\frac{n}{n-2}$.}
In particular, if $M$ satisfies $BG(0)$, then the constant $c^*(n,R)$ in Proposition 5.3 and 5.4 can be chosen depending only on $n$.
 We refer reader to Chapter 4 in \cite{Hei} for the proofs of Proposition 5.3 and 5.4.

We remark that Proposition  5.4 actually  holds for any $f\in W_0^{1,2}(B_p(R))$, because $DC_0(B_p(R))$ is dense in
$W_0^{1,2}(B_p(R))$ (see \cite{KMS}).

By applying Poincar\'e and Sobolev inequalities and the standard Nash-Morse iteration, one has the following mean value theorem.
\begin{prop}
Let $M$ be an $n-$dimensional Alexandrov space  satisfying $BG(k).$  Then there exists a constant
$C=C(n,R)>0$
such that for any nonnegative subharmonic function $u$, we have
\be{equation}{\sup_{B_p(R/2)} u\ls  C\cdot \frac{1}{{\rm vol}B_p(R)}\int_{B_p(R)}u.}
Moreover, if $M$ satisfies $BG(0)$, then the constant $C=C(n).$
\end{prop}

In \cite{KMS}, K. Kuwae, Y. Machigashira and T. Shioya  proved that  the induced distance by a canonical Dirichlet form is equivalent to original one for Alexandrov spaces (see Theorem 7.1 in \cite{KMS}).  By combining Sturm's work on strongly local regular Dirichlet form in \cite{S-diff,S-diri}, they obtained that  the existence of the heat kernel, H\"older continuity and the parabolic Harnack inequality for the solutions of the heat equation on Alexandrov spaces. Furthermore, they proved,
\be{prop}{$($Theorem 1.5 in \cite{KMS}$)$\indent Let $M$ be an $n-$dimensional Alexandrov space  satisfying $BG(k)$.
 For any bounded open set $\Omega\subset M$. Denote by $\lambda_1\ls \lambda_2 \ls\cdots$ all the eigenvalues of $\Delta^\Omega$ with multiplicity and by $\{u_i\}_{i=1}^\infty$ the sequence of  associated eigenfunctions which is a complete orthonormal basis of $L^2(\Omega)$. Let $p^\Omega(t,x,y)$ be the heat kernel on $\Omega$. Then we have \be{equation}{p^\Omega(t,x,y)=\sum_{i=1}^\infty e^{-\lambda_it}u_i(x)u_i(y)} for any $t>0$ and $x,y\in \Omega$, where the convergence is uniform on any compact subset of $(0,\infty)\times\Omega\times\Omega.$
}

In \cite{KMS}, it was  proved that the  first eigenvalue on a
bounded open set of an  Alexandrov space  is positive.  For higher
eigenvalues, one has the following lower bound estimates.
\be{prop}{Let $M$ be an $n-$dimensional Alexandrov space satisfying
$BG(k)$. Given constants $D$, there exists $C=C(n,k,D)$ such that
the $j-$th eigenvalue satisfies $$\lambda_j(\Omega)\gs C\cdot
j^{2/n},\qquad j\gs1$$ for all open sets $\Omega $ with ${\rm
diam}(\Omega)\ls D$. }Its proof follows exactly as smooth case by
using  Bishop-Gromov volume comparison and  Sobolev inequality (see,
for example, in \cite{SY}).

On the other hand, the standard proof for the H\"older continuity of harmonic functions  implies a stronger Liouville property: \be{prop}{$($see Corollary 3.4 in \cite{S-diri}$)$\indent Let $M$ be an $n-$dimensional Alexandrov space  satisfying $BG(0)$, then all of positive harmonic functions on $M$ are constants.}

Let $M^n$ be an $n-$dimensional Riemannian manifold with nonnegative Ricci curvature, Yau in \cite{Y} conjectured that the space of harmonic functions on $M^n$ with at most polynomial growth of degree $d$ must be finite dimension for any $d\in \R^+$. The conjecture was proved by Colding and Minicozzi in \cite{CM1}. In fact, they proved a more general statement which only assumed that $M^n$ admits a doubling property and a Poincar\'e inequality. Later in \cite{Li}, Peter Li gave a short proof under weaker conditions that the manifold $M^n$ admits a doubling property and a mean value inequality.

Recently, Hua in \cite{Hua} extended Colding-Minicozzi's argument in \cite{CM2} to prove the following result.
\be{prop}{$($Hua \cite{Hua}$)$\indent Let $M$ be an $n-$dimensional  Alexandrov space with nonnegative curvature. $\partial M=\varnothing.$ Then the space of
harmonic functions with polynomial growth of degree $\ls d$  is finite dimensional for any $d\in\R^+$.}
After obtaining Proposition 5.3, one can actually prove the above proposition under the weaker condition  $BG(0)$ (replacing the nonnegative curvature condition).

In the end of the subsection, let us consider the Gaussian estimates for heat kernel under condition $BG(k)$. Let $M$ be an $n-$dimensional Alexandrov space satisfying $BG(k)$ and let $\Omega$ be an open set $\Omega\subset M$. If $k<0$ we add the assumption that $\Omega$ is bounded.

Let $C_D$ and $C_P$ be the double constant and Poincar\'e constant (Proposition 5.2) in $\Omega$. If $M$ satisfies $BG(0)$, both constants depend only on the dimension of $M$. In general, they depend also on  the diameter of $\Omega.$

By combining Theorem 7.1, 8.3 in \cite{KMS} and Theorem 4.1, 4.8 in \cite{S-diri}, we have the following Gaussian type bounds for heat kernel on $\Omega$:
\begin{thm}$($\cite{KMS,S-diri}$)$ There exists a constant $C$ depending only on $C_D$ and $C_P$ such that the following estimates hold true
\be{equation} {p(t,x,y)\ls C\cdot \big({\rm vol}B_x(\sqrt t)\cdot {\rm vol}B_y(\sqrt t)\big)^{-1/2}\cdot\exp\big(-\frac{|xy|^2}{5t}\big)} for all
$x,y\in \Omega$ and $\sqrt t<\min\{{\rm dist}(x,\partial \Omega),{\rm dist}(x,\partial \Omega)\},$ and
\be{equation} {p(t,x,y)\gs C^{-1}\cdot \big({\rm vol}B_x(\sqrt t)\big)^{-1}\cdot\exp\big(-\frac{C|xy|^2}{t}-\frac{Ct}{R^2}\big)} for all
$x,y\in \Omega$ which are joined in $\Omega$ by a curve $\gamma\subset\Omega$. Here $\sqrt t<R^2$ with $R={\rm dist}(\gamma,\partial \Omega)$.
\end{thm}

\subsection{Lipschitz continuity for  heat kernel}
Petrunin in \cite{Pet-har} sketched a proof to the Lipschitz continuity of harmonic functions on Alexandrov spaces.
\begin{thm}$($Petrunin \cite{Pet-har}$)$\indent
Let $\Omega\subset M$ be an open domain in an $n-$dimensional Alexandrov space $M$ with curvature $\gs \ka$ on $\Omega$.  Then, for any compact subset $K\subset\Omega$ with ${\rm diam}(K)\ls D,\ {\rm vol}(K)\gs v$ and ${\rm dist}(K,\partial\Omega)\ge\rho$, there exists a positive constant $L=L(n,\ka,D,v,\rho)$ such that
 \be{equation}{|\nabla f|_{L^\infty(K)}\ls L\cdot\|f\|_{W^{1,2}(\Omega)},}
for all harmonic functions $f:\Omega\to \R$. \end{thm}

A similar regularity  problem  for harmonic maps between singular
spaces was also studied. Korevaar-Schoen \cite{KoS} proved that a
harmonic map from a smooth Riemannian manifold to a non-positive
curved space (in sense of Alexandrov) is locally Lipschitz. Lin
\cite{Lin} proved that a harmonic map from an Alexandrov space  to a
non-positive curved  space (in sense of Alexandrov) is H\"older
continuous. However, the following question arisen by F.H.Lin in
\cite{Lin} is still open.\begin{conj}$(Lin$ \cite{Lin}$)$\indent A
harmonic map from an Alexandrov space  to a non-positive curved
space (in sense of Alexandrov) is locally Lipschitz.\end{conj}

As an consequence of Petrunin's estimate for harmonic functions, we
can obtain an interior estimate for gradient of eigenfunctions as
follows.
\begin{prop}
Let $\Omega\subset M$ be an open domain in an $n-$dimensional Alexandrov space $M$ with curvature $\gs \ka$ on $\Omega$. Let
$f$ be an eigenfunction on $\Omega$ with respect to eigenvalue of the Laplacian $\lambda$ and $\|f\|_2=1.$
 Then for any compact subset $K\subset\Omega$ with ${\rm diam}(K)\ls D,\ {\rm vol}(K)\gs v$ and ${\rm dist}(K,\partial\Omega)\ge\rho$ we have the estimate
\be{equation}{|\nabla f|_{L^\infty(K)}\ls L\cdot\sqrt{\frac{2\lambda+1}{2\sqrt\lambda}}\cdot e^{\sqrt\lambda\cdot\rho}}
and \be{equation}{|f(x)|\ls v^{-1/2}+DL\cdot\sqrt{\frac{2\lambda+1}{2\sqrt\lambda}}\cdot e^{\sqrt\lambda\cdot\rho}.
}
where constant $L=L(n,\ka,D,v,\rho)$. \end{prop}
\be{proof}
{Consider function $w(x,t)=e^{\sqrt\lambda t}\cdot f(x)$ in $\Omega\times (0,D+2\rho)$. Clearly it is  harmonic.
 By setting  $K_1=K\times[\rho,D+\rho]$ and $I=(0,D+2\rho)$, we have
 \be{equation*}{|\nabla w|_{L^\infty(K_1)}\ls L(\sqrt2D,Dv,n,k\wedge0,\rho)\cdot\|w\|_{W^{1,2}(\Omega\times I)}.}
Noted that $$|\nabla w|_{L^\infty(K_1)}\gs e^{\sqrt\lambda(D+\rho)}\cdot|\nabla f|_{L^\infty(K)}$$
and $$\|w\|_{W^{1,2}(\Omega\times I)}^2=(2\lambda+1)\int_Ie^{2\sqrt \lambda t}dt\ls  e^{2\sqrt\lambda (D+2\rho)}\cdot\frac{2\lambda+1}{2\sqrt\lambda}.$$
Then the desired estimate (5.10) holds. By the assumption that
${\rm vol}K\gs v$, we get $$\min_Kf^2\cdot v\ls \int_Kf^2\ls \int_\Omega f^2=1.$$
So $\min_Kf\ls 1/\sqrt{v}.$
Thus the desired estimate (5.11) follows from this and the gradient estimate  (5.10).}

Now by combining  Proposition 5.6 and the above estimates for eigenvalues and eigenfunctions above, we can prove that the heat kernel is  locally Lipschitz continuous.

\begin{thm}Let $\Omega\subset M$ be an open domain in an $n-$dimensional Alexandrov space $M$ with curvature $\gs \ka$ on $\Omega$. Let
$p(t,x,y)$ be the heat kernel on $\Omega$.
 Then for any compact subset $K\subset\Omega$ with ${\rm diam}(K)\ls D,\ {\rm vol}(K)\gs v$ and ${\rm dist}(K,\partial\Omega)\ge\rho$ we have the estimate
\be{equation}{|\nabla_y p(t,x,y)|_{L^\infty([\rho^2,+\infty)\times K\times K)}\ls c(n,\ka,D,v,\rho)}
for some constant $ c(n,\ka,D,v,\rho)$. \end{thm}
\begin{proof}
By Proposition 5.6, 5.7 and 5.13,  we get\be{equation*}{\be{split}{
|\nabla_y&p(t,x,y)|\\ &\ls \sum_{j=1}^\infty e^{-\lambda_j\cdot\rho^2}\cdot\Big(v^{-1/2}+DL\cdot\sqrt{\frac{2\lambda_j+1}{2\sqrt{\lambda_j}}}\cdot e^{\sqrt\lambda_j\cdot\rho}\Big)\cdot\Big(L\cdot\sqrt{\frac{2\lambda_j+1}{2\sqrt{\lambda_j}}}\cdot e^{\sqrt{\lambda_j}\cdot\rho}\Big)}
} for all $(t,x,y)\in[\rho^2,+\infty)\times K\times K.$

Denote by $$N_1=\Big\{j\in\mathbb N:\ \sqrt{\lambda_j}\ls \max\{1,\ 4/\rho,\ -\rho^{-1}\cdot\ln(DLv^{1/2})\}\Big\}.$$
By Proposition 5.7, we have $$\# N_1\ls j_0(n,\ka,D,v,\rho),$$ and hence
\begin{equation*}\begin{split}I:&=\sum_{j\in N_1} e^{-\lambda_j\cdot\rho^2}\cdot\Big(v^{-1/2}+DL\cdot\sqrt{\frac{2\lambda_j+1}{2\sqrt{\lambda_j}}}\cdot e^{\sqrt\lambda_j\cdot\rho}\Big)\cdot\Big(L\cdot\sqrt{\frac{2\lambda_j+1}{2\sqrt{\lambda_j}}}\cdot e^{\sqrt{\lambda_j}\cdot\rho}\Big)\\&\ls c_1(n,\ka,D,v,\rho)\end{split} \end{equation*}
for some constant $c_1(n,\ka,D,v,\rho).$
By the definition of the set $N_1$, we get
\begin{equation*}\begin{split}II:&=\sum_{j\not\in N_1} e^{-\lambda_j\cdot\rho^2}\cdot\Big(v^{-1/2}+DL\cdot\sqrt{\frac{2\lambda_j+1}{2\sqrt{\lambda_j}}}\cdot e^{\sqrt\lambda_j\cdot\rho}\Big)\cdot\Big(L\cdot\sqrt{\frac{2\lambda_j+1}{2\sqrt{\lambda_j}}}\cdot e^{\sqrt{\lambda_j}\cdot\rho}\Big)\\&\ls
\sum_{j\not\in N_1} e^{-\lambda_j\cdot \rho^2}\cdot3DL^2\cdot\sqrt{\lambda_j}\cdot e^{2\sqrt{\lambda_j}\rho}\\&\ls 3DL^2\sum_{j\not\in N_1}e^{-\lambda_j\cdot \rho^2/2}\cdot\sqrt{\lambda_j}.
\end{split} \end{equation*}
By applying Proposition 5.7 again,  we have  $$\sqrt{\lambda_j}\cdot e^{-\lambda_j\cdot\rho^2/2}\ls \sqrt{\lambda_j}\cdot(2n)!\cdot(\lambda_j\cdot\rho^2/2)^{-2n}\ls c_2(n,\ka,D,\rho)\cdot j^{-(4-1/n)}$$ for all $j\in\mathbb N.$
Thus, we have $$II\ls c_3(n,\ka,D,v,\rho).$$
Therefore, the proof is completed.\end{proof}

\end{document}